\documentclass[12pt]{amsart}
\usepackage[margin=1.25in]{geometry}
\usepackage{amsmath,amssymb,amsthm,amscd}
\usepackage{mathrsfs}
\usepackage{mathtools}
\usepackage{enumerate}
\usepackage[T1]{fontenc}
\usepackage{lmodern}
\usepackage{microtype}
\usepackage{hyperref}
\hypersetup{colorlinks=true,linkcolor=blue,citecolor=blue,urlcolor=blue}

\def\ge{\geqslant}
\def\le{\leqslant}
\def\phi{\varphi}
\def\epsilon{\varepsilon}
\def\tilde{\widetilde}

\def\mapsto{\longmapsto}

\DeclareMathOperator{\Ext}{Ext}

\DeclareMathOperator{\Spec}{Spec}

\DeclareMathOperator{\ord}{ord}

\newcommand{\m}{\mathfrak m}
\newcommand{\sO}{\mathcal O}

\newcommand{\reg}{\operatorname{reg}}

\theoremstyle{plain}
\newtheorem{thm}{Theorem}[section]
\newtheorem{cor}[thm]{Corollary}
\newtheorem{prop}[thm]{Proposition}
\newtheorem{lem}[thm]{Lemma}
\newtheorem{theo}{Theorem}

\newtheorem{COR}[theo]{Corollary}
\theoremstyle{definition}
\newtheorem{defn}[thm]{Definition}

\newtheorem{notation}[thm]{Notation}
\theoremstyle{remark}
\newtheorem{rem}[thm]{Remark}
\makeatletter

\@addtoreset{equation}{thm}
\makeatother

\title[Frobenius liftability of complete intersections]
{Frobenius liftability of isolated complete intersections}
\author{Tatsuro Kawakami}
\address{Graduate School of Mathematical Sciences, University of Tokyo,
3-8-1 Komaba, Meguro-ku, Tokyo 153-8914, Japan}
\email{kawakami@ms.u-tokyo.ac.jp}
\author{Teppei Takamatsu}
\address{Department of Mathematics, Faculty of Science, Saitama University,
255 Shimo-Okubo, Sakura-ku, Saitama-shi, Saitama 338-8570, Japan}
\email{teppeitakamatsu.math@gmail.com}
\date{}
\begin{document}
\begin{abstract}
We classify Frobenius liftable normal isolated complete intersection
singularities of dimension at least three over an algebraically closed
field of positive characteristic. In dimension three, they are precisely
the regular points and the ordinary double points; in higher dimensions,
they are regular. 
%The proof applies in every positive characteristic,
%including characteristic two, and does not use a classification of
%three-dimensional singularities.
We also characterize Frobenius liftability of the complete local
hypersurfaces defined by $xy-g(z,t)$ without assuming that the singularity
is isolated. Such a hypersurface is Frobenius liftable if and only if,
after a formal change of coordinates, $g$ is a unit times
$H^pz^at^b$ with $0\le a,b<p$.
\end{abstract}
\maketitle
\markboth{Tatsuro Kawakami and Teppei Takamatsu}
{Frobenius liftability of complete intersections}
\setcounter{tocdepth}{1}
%\tableofcontents

\section{Introduction}\label{sec:introduction}
We work over an algebraically closed field $k$ of characteristic $p>0$,
unless stated otherwise.
A complete local $k$-algebra $R$ is said to be \textit{$F$-liftable} if
there exists a flat lifting of $R$ over the ring $W_2(k)$ of Witt vectors
of length two together with a lifting of its Frobenius endomorphism.
The lifting of the ring is allowed to vary, and the lifted
Frobenius acts on $W_2(k)$ as its Frobenius morphism.

Frobenius liftability imposes strong restrictions on singularities.
In particular, a lifting of Frobenius gives split injections from
differential forms into their Frobenius pushforwards
\cite{BTLM}.
For singular spaces, these maps extend to reflexive differentials.
Zdanowicz \cite{Zda18} developed computational criteria for local
Frobenius liftability and proved, in odd characteristic, the
non-liftability of ordinary double points of dimension at least four.
For surface singularities, the relationship between $F$-liftability and
$F$-purity is substantially less restrictive; we refer to \cite{KT,KT2}
for the classification in the surface case.

Our first result gives a classification for isolated complete
intersections in dimension at least three.

\begin{theo}\label{thm:main-ici}
Let $(R,\m)$ be a normal complete local complete intersection
$k$-algebra with residue field $k$ and dimension $d\ge3$.
Suppose that $\Spec R\setminus\{\m\}$ is regular.
Then the following are equivalent:
\begin{enumerate}[(1)]
\item $R$ is $F$-liftable.
\item Either $R$ is regular, or $d=3$ and
\[
R\cong k[[x,y,z,t]]/(xy-zt).
\]
\end{enumerate}
\end{theo}

Thus, a singular isolated three-dimensional hypersurface is
$F$-liftable precisely when it is an ordinary double point.
In particular, isolated compound Du Val singularities other than the
ordinary double point are excluded, including those of types
$cD$ and $cE$.
%The theorem has no restriction on the characteristic.

In characteristic zero, Kawakami and Witaszek
\cite[Theorem B and Theorem 4.2]{KW} obtained a related classification
for isolated rational hypersurface singularities whose reductions are
$F$-liftable over a Zariski-dense set of closed points of a model.
Theorem \ref{thm:main-ici} concerns a single positive characteristic
and also allows complete intersections of arbitrary embedding codimension.

Since three-dimensional Gorenstein terminal singularities are isolated complete intersection singularities by \cite{Tanaka25},
we obtain the following corollary:

\begin{COR}
Let $(R,\mathfrak m)$ be the completed local ring at a closed point
of a three-dimensional Gorenstein terminal variety over $k$.
Then the following are equivalent:
\begin{enumerate}[(1)]
\item $R$ is $F$-liftable.
\item Either $R$ is regular or 
\[
R\cong k[[x,y,z,t]]/(xy-zt).
\]
\end{enumerate}
\end{COR}
Three-dimensional Gorenstein terminal singularities are strongly $F$-regular, and hence in particular $F$-pure, when $p>5$ (Proposition~\ref{prop:terminal-cdv-fregular}). It is worth emphasizing the sharp contrast with the two-dimensional case: while $F$-purity and $F$-liftability are nearly indistinguishable in dimension two, $F$-liftable singularities form a much more restrictive class in higher dimensions.

The main point of the proof of Theorem \ref{thm:main-ici} is a restriction on the first cotangent
cohomology module.
For an isolated complete intersection of dimension $d\ge3$, the
module $\Omega_R^1$ is reflexive and admits a two-term free resolution.
Frobenius liftability therefore gives a split injection
\[
H^{d-1}_{\m}(\Omega_R^1)\longrightarrow
F_*H^{d-1}_{\m}(\Omega_R^1).
\]
This local cohomology module has finite length.
An elementary observation about injective Frobenius-semilinear maps
then shows that it is annihilated by $\m$.
By local duality, the same holds for
$\Ext_R^1(\Omega_R^1,R)$.
A count of minimal generators forces the embedding codimension to
be at most one.
In the hypersurface case, the Jacobian ideal is consequently the maximal
ideal. We deduce that the equation is formally equivalent to a quadratic
form with nondegenerate polar form.
An explicit obstruction calculation then excludes the higher-dimensional
quadrics.

Our second result treats equations of the form $xy-g(z,t)$ without an
isolatedness assumption.
Here and below, the phrase \textit{of $cA$ type} refers to this
form of the equation (cf.\ Remark~\ref{rem:cA}).

\begin{samepage}
\begin{theo}\label{thm:main-ca}
Let $B=k[[z,t]]$, let $0\ne g\in(z,t)$, and set
\[
R_g=B[[x,y]]/(xy-g).
\]
Then the following are equivalent:
\begin{enumerate}[(1)]
\item $R_g$ is $F$-liftable.
\item After a formal change of coordinates in $B$, one can write
\[
g=uH^pz^at^b,\qquad
u\in B^\times,\quad 0\ne H\in B,\quad 0\le a,b<p.
\]
\item The reduced curve formed by the irreducible factors of $g$
whose multiplicities are not divisible by $p$ is empty, smooth,
or the union of two smooth branches meeting transversely.
\end{enumerate}
\end{theo}
\end{samepage}

The factor $H^p$ is essential in this statement.
For example, $xy-H^p$ is $F$-liftable for every nonzero
$H\in(z,t)$, regardless of the singularities of the curve $H=0$.
In contrast, when $g$ is reduced, the only possibilities are a smooth
curve and an ordinary node.
For $g\in(z,t)^2$ with $R_g$ an isolated singularity,
Theorem \ref{thm:main-ca} thus recovers the ordinary double point.

% To prove Theorem \ref{thm:main-ca}, we first reduce the hypersurface
% criterion to a condition in the two-variable power series ring $B$.
% This condition produces a Frobenius lifting on a lifting of $B$ which
% preserves a lifting of $g$ up to a prescribed unit.
% Its divided differential has determinant divisible by $D^{p-1}$,
% where $D$ is the reduced product of the factors with multiplicity
% prime to $p$.
% The Cartier identity for this determinant implies
% $D^{p-1}\notin(z^p,t^p)$.
% A direct calculation with the tangent cone shows that $D=0$ is
% smooth or nodal, unless it is empty.
% The converse follows from an explicit Frobenius lifting.

\subsection*{Notation and conventions}
Throughout this paper, $k$ is an algebraically closed field of characteristic
$p>0$.
All rings are Noetherian.
A complete local $k$-algebra is understood to have residue field $k$
and finite embedding dimension.
We write $\m_R$ for its maximal ideal and $F_*M$ for the abelian
group $M$ endowed with the $R$-module structure
$a\cdot F_*m=F_*(a^pm)$.
The notation $\Omega_R^1$ denotes differentials over $k$.
For power series presentations, these can be computed by continuous
differentials; the two descriptions agree in our setting, as explained
below.
For $U=(\Spec R)_{\reg}$ and $j\colon U\hookrightarrow\Spec R$,
we write $\Omega_R^{[i]}=j_*\Omega_U^i$, identifying coherent sheaves
on an affine scheme with their modules.
An isomorphism of complete local rings and a change of coordinates
always mean a continuous $k$-isomorphism.

\section{Differentials and Frobenius splittings}\label{sec:preliminaries}

We write $W_2(k)$ for the ring of Witt vectors of length two and
$\sigma\colon W_2(k)\to W_2(k)$ for its Frobenius automorphism.

\begin{defn}\label{pre:definition}
A complete local $k$-algebra $R$ is \emph{Frobenius liftable}, or
\emph{$F$-liftable}, if there exist a flat complete local $W_2(k)$-algebra
$\widetilde R$, an identification $\widetilde R/p\widetilde R\simeq R$,
and a continuous homomorphism
$\widetilde F\colon\widetilde R\to\widetilde R$ such that
$\widetilde F|_{W_2(k)}=\sigma$ and the reduction of $\widetilde F$ modulo
$p$ is the absolute Frobenius of $R$.
\end{defn}

\subsection{Differentials of complete local rings}

For a complete local $k$-algebra $R$, we denote its module
\[
\widehat\Omega^1_{R}\coloneqq \varprojlim_n \Omega_R^1/\m^n \Omega_R^1
\]
of continuous
$k$-differentials by $\Omega_R^1$. In the present setting this agrees
with the ordinary module of K\"ahler differentials. Indeed, if
$S=k[[x_1,\ldots,x_n]]$, then every element $a\in S$ has a unique
expansion
\[
 a=\sum_{0\leq I_1,\ldots,I_n<p}x^I a_I^p,
 \qquad a_I\in S,
\]
because $k$ is perfect. Every $k$-derivation therefore satisfies
\[
 da=\sum_I a_I^p\,d(x^I)
    =\sum_{j=1}^n\frac{\partial a}{\partial x_j}\,dx_j.
\]
Consequently, $\Omega^1_{S/k}=\bigoplus_{j=1}^n S\,dx_j$.
Passing to a quotient $R=S/I$ gives
\begin{equation}\label{pre:conormal}
 I/I^2\longrightarrow\bigoplus_{j=1}^n R\,dx_j
 \longrightarrow\Omega_R^1\longrightarrow0,
 \qquad f\longmapsto\sum_j\frac{\partial f}{\partial x_j}\,dx_j.
\end{equation}
This is also the presentation of continuous differentials.

We use continuous relative differentials for liftings over $W_2(k)$.
Their reduction modulo $p$ is the module just described. If $R$ is
normal of dimension $d$ and $U\subseteq\Spec R$ is its regular locus,
the formal Jacobian criterion shows that $\Omega_U^1$ is locally free
of rank $d$. The same criterion, applied to a flat lifting and an
invertible Jacobian minor, shows that continuous differentials on the
corresponding open subset $\widetilde U$ of the lifting are locally free
of rank $d$. 
% In particular they are flat over $W_2(k)$.
% Here the rank is the dimension of the original complete local ring,
% not the dimension of a localization at an individual point.

% The usual Cartier isomorphism applies on $U$. One way to see this
% without a finite-type assumption is to use local finite $p$-bases:
% the formal Jacobian criterion provides $p$-bases on the regular locus,
% and the standard calculation of the de Rham complex in a $p$-basis
% is unchanged. Thus, writing
% \[
%  Z\Omega_U^i:=\ker\bigl(d\colon F_*\Omega_U^i
%                   \longrightarrow F_*\Omega_U^{i+1}\bigr),
% \]
% we have the Cartier operator
% $C\colon Z\Omega_U^i\to\Omega_U^i$.
Since $k$ is perfect, differentials over $k$ agree with
absolute differentials over $\mathbb F_p$.
The Cartier isomorphism for regular Noetherian
$\mathbb F_p$-algebras \cite[Theorem~9.5.6]{BLM}
therefore applies on $U$.
We use its multiplicativity and the following identities,
where $a$ is a local function and $\eta$ is a closed form:
\[
 C(da)=0,\qquad C(a^{p-1}da)=da,\qquad C(a^p\eta)=aC(\eta).
\]
% Here and below, $F_*M$ denotes the underlying additive group of an
% $R$-module $M$ endowed with the scalar action
% $a\cdot F_*m=F_*(a^pm)$.

\subsection{A finite-length lemma}

The following elementary observation is the finite-length case of the
radical-annihilator property for modules with injective Frobenius
\cite[Proposition~3.5(a)]{EH08}.

\begin{lem}\label{pre:finite-length}
Let $(R,\m)$ be a local ring of characteristic $p>0$, and let $N$ be
an $R$-module of finite length. Suppose that an injective additive map
$\varphi\colon N\to N$ satisfies
\[
 \varphi(au)=a^p\varphi(u)
 \qquad\text{for all }a\in R\text{ and }u\in N.
\]
Then $\m N=0$.
\end{lem}

\begin{proof}
Choose an integer $q\geq1$ such that $\m^qN=0$, and then choose $e\geq1$
with $p^e\geq q$. For $a\in\m$ and $u\in N$, semilinearity gives
\[
 \varphi^e(au)=a^{p^e}\varphi^e(u)=0.
\]
The iterate $\varphi^e$ is injective, so $au=0$. This holds for every
$a\in\m$ and $u\in N$.
\end{proof}

\subsection{Splitting reflexive one-forms}

% Let $R$ be a normal complete local $k$-algebra, let $X=\Spec R$, and let
% $j\colon U\hookrightarrow X$ be the inclusion of the regular locus.
% We write $\Omega_R^{[i]}$ for the finite reflexive module corresponding
% to $j_*\Omega_U^i$.
% The splitting below is the local form of the construction in
% \cite[Theorem 2 and Section 2.3]{BTLM}. 

\begin{prop}[{\cite[Theorem 2 and Section 2.3]{BTLM}}]\label{pre:splitting}
Let $R$ be a normal complete local $k$-algebra of dimension $d\geq1$.
If $R$ is $F$-liftable, then there is a split $R$-linear injection
\[
 \Omega_R^{[1]}\longrightarrow F_*\Omega_R^{[1]}.
\]
\end{prop}

\begin{proof}
The proof of \cite[Theorem 2 and Section 2.3]{BTLM}
works equally well in our complete local setting,
using continuous differentials $\widehat\Omega^1_{\tilde{R}/W_2(k)}$ on the lifting $\tilde{R}$ of $R$.
\end{proof}

\section{A hypersurface criterion}\label{sec:criterion}
We record the hypersurface criterion in the complete local setting \cite[Corollary 4.12]{Zda18}.
% It is the formal power series version of
% \cite[Corollary 4.12]{Zda18}.
% We include the proof to specify the choices of liftings and the sign
% convention used later.

\begin{notation}\label{not:deltaone-definitions}
Let $S=k[[x_1,\ldots,x_n]]$ and
$\widetilde S=W_2(k)[[x_1,\ldots,x_n]]$.
We denote by $\Phi_0$ the continuous endomorphism of
$\widetilde S$ that restricts to the Witt Frobenius on $W_2(k)$
and satisfies $\Phi_0(x_i)=x_i^p$.
For $f=\sum_{\alpha\in\mathbb N^n}a_\alpha x^\alpha\in S$,
following \cite[Convention~4.1]{KT}, we set
\[
\Delta_1(f)
:=\sum_{\substack{\sum_\alpha i_\alpha=p\\0\leq i_\alpha<p}}
\overline{\frac{1}{p}\binom{p}{(i_\alpha)_\alpha}}
\prod_\alpha(a_\alpha x^\alpha)^{i_\alpha}.
\]
Here each $(i_\alpha)_\alpha$ has finite support.
Each indicated multinomial coefficient is divisible by $p$;
the bar denotes reduction in $k$ after division over $\mathbb Z$.
The sum is interpreted in the $(x_1,\ldots,x_n)$-adic topology.
For a polynomial this is precisely the mixed multinomial
expression of \cite[Convention~4.1]{KT}; the same expression
is well-defined for formal power series, since only finitely
many terms contribute in any fixed total degree.

To relate this notation to liftings, let
\[
\widetilde f_{\mathrm{can}}
:=\sum_{\alpha\in\mathbb N^n}[a_\alpha]x^\alpha
\in\widetilde S,
\]
where $[a]$ denotes the Teichm\"uller representative of $a\in k$.
Since $\Phi_0([a_\alpha]x^\alpha)=
[a_\alpha]^p x^{p\alpha}$, the multinomial expansion gives
\[
\Delta_1(f)
=\frac{\widetilde f_{\mathrm{can}}^{\,p}
              -\Phi_0(\widetilde f_{\mathrm{can}})}{p}
\bmod p.
\]
Division by $p$ means the inverse of the isomorphism of
$S$-modules
\[
S=\widetilde S/p\widetilde S
\xrightarrow{\ \sim\ }p\widetilde S,
\qquad \overline U\longmapsto pU.
\]
If $F=\widetilde f_{\mathrm{can}}+p\widetilde h$ is any lifting
of $f$, with $h=\widetilde h\bmod p$, then
\[
\frac{\Phi_0(F)-F^p}{p}\bmod p
=-\Delta_1(f)+h^p.
\]
Indeed, $F^p=\widetilde f_{\mathrm{can}}^{\,p}$ in
$\widetilde S$ and
$\Phi_0(p\widetilde h)=p\Phi_0(\widetilde h)$.
\end{notation}

% 2. Replacement for Lemma 3.2 and its proof.
\begin{lem}[{\cite[Theorem~4.2]{KT};
\cite[Corollary~4.12]{Zda18}}]\label{lem:deltaone-criterion}
Let $0\neq f\in(x_1,\ldots,x_n)S$.
Then $S/(f)$ is $F$-liftable if and only if there exists
$h\in S$ such that
\[
\Delta_1(f)+h^p
\in(f,f_{x_1}^{\,p},\ldots,f_{x_n}^{\,p})S.
\tag{\thethm.1}\label{eq:deltaone-hypersurface-criterion}
\]
\end{lem}

\begin{proof}
The hypersurface criterion of
\cite[Corollary~4.12]{Zda18} applies equally to formal power
series rings.  
% To identify its obstruction with the notation
% above, write an arbitrary lifting of the equation as
% $F=\widetilde f_{\mathrm{can}}+p\widetilde u$ and a Frobenius
% lifting on $\widetilde S$ as
% $\Phi(x_i)=x_i^p+p\widetilde a_i$.
% Put $u=\widetilde u\bmod p$ and
% $a_i=\widetilde a_i\bmod p$.
% Taylor expansion modulo $p^2$ gives
% \[
% \frac{\Phi(F)-F^p}{p}\bmod p
% =-\Delta_1(f)+u^p
%  +\sum_{i=1}^n a_i f_{x_i}^{\,p}.
% \]
% Since $\widetilde S/(F)$ is flat over $W_2(k)$, the
% Frobenius lifting descends to this quotient precisely
% when the right-hand side belongs to $(f)$.
% Thus the criterion is equivalent to the existence of $u\in S$
% such that
% $-\Delta_1(f)+u^p\in(f,f_{x_1}^{\,p},\ldots,f_{x_n}^{\,p})S$.
% Negating this containment and setting $h=-u$ gives
% \eqref{eq:deltaone-hypersurface-criterion}.
\end{proof}

\section{Isolated complete intersections}\label{sec:isolated}

\subsection{The cotangent module and the embedding codimension}

We first isolate the consequence of Proposition~\ref{pre:splitting}
that controls the defining equations.

\begin{samepage}
\begin{prop}\label{iso:codimension}
Let $(R,\m)$ be a normal complete local complete intersection
$k$-algebra of dimension $d\geq3$, regular away from its closed point.
Choose a minimal presentation
\[
 R=S/(f_1,\ldots,f_c),\qquad
 S=k[[x_1,\ldots,x_n]],\qquad f_i\in\m_S^2,
\]
where $f_1,\ldots,f_c$ is a regular sequence and $n=d+c$.
Let
\[
 J=\left(\frac{\partial f_i}{\partial x_j}\right)_{\substack{1\leq i\leq c\\1\leq j\leq n}}
 \colon R^n\longrightarrow R^c
\]
be the Jacobian matrix. If $R$ is $F$-liftable, then
\begin{equation}\label{iso:jacobian-image}
 \operatorname{Im}J=\m R^c.
\end{equation}
Consequently, $c\leq1$.
\end{prop}
\end{samepage}

\begin{proof}
There is nothing to prove if $c=0$, so assume $c\geq1$. 
By the conormal exact sequence \eqref{pre:conormal},  we have
\begin{equation}\label{iso:resolution}
 0\longrightarrow R^c\xrightarrow{J^{\mathsf t}}R^n
 \longrightarrow\Omega_R^1\longrightarrow0,
\end{equation}
where the first map is generically injective and thus it is injective since $R^c$ is torsion-free.
This exact sequence shows that $\Omega^1_R$ is reflexive, and
$\Omega_R^1=\Omega_R^{[1]}$.

By Proposition~\ref{pre:splitting}, the map
$\Omega^1_R\to F_*\Omega^1_R$ is a split injection,
and therefore we obtain an injective $R$-linear map
\[
 H^{d-1}_{\m}(\Omega^1_R)\longrightarrow H^{d-1}_{\m}(F_*\Omega^1_R),
\]
which is a Frobenius-semilinear
injection as in Lemma~\ref{pre:finite-length}.

Let $E=E_R(k)$ be the injective hull of the residue field and write
$D(-)=\operatorname{Hom}_R(-,E)$. 
Set $N:=H_{\mathfrak m}^{d-1}(\Omega_R^1)$.
Since $R$ is complete and
Gorenstein, local duality and \eqref{iso:resolution} yield
\begin{equation}\label{iso:duality}
 D(N)\simeq\operatorname{Ext}^1_R(\Omega_R^1,R)
 \simeq T_R^1:=\operatorname{Coker}(J\colon R^n\to R^c).
\end{equation}
The module $\Omega_R^1$ is locally free away from $\m$, so $T_R^1$ is a
finite module supported at $\m$ and hence has finite length.
Matlis duality then implies that $N$ has finite length as well.
Lemma~\ref{pre:finite-length} gives $\m N=0$. Dualizing, we obtain
\[
 \m T_R^1=0.
\]
Equivalently, $\m R^c\subseteq\operatorname{Im}J$.
Every entry of $J$ belongs to $\m$, since $f_i\in\m_S^2$, so the
opposite inclusion also holds. This proves
\eqref{iso:jacobian-image}.

The image of $J$ is generated by its $n$ columns. On the other hand,
minimality of the presentation gives
$\dim_k(\m/\m^2)=n$, and therefore
\[
 \dim_k\frac{\m R^c}{\m(\m R^c)}
 =c\dim_k\frac{\m}{\m^2}=cn.
\]
By Nakayama's lemma, $cn$ is the minimal number of generators of
$\m R^c$. Thus $cn\leq n$, and $c\leq1$.
\end{proof}

\begin{cor}\label{iso:jacobian-hypersurface}
Let $S=k[[x_1,\ldots,x_n]]$ and $R=S/(f)$ be a normal isolated
hypersurface singularity of dimension $n-1\geq3$, with $f\in\m_S^2$.
If $R$ is $F$-liftable, then
\[
 J_f:=\left(\frac{\partial f}{\partial x_1},\ldots,
            \frac{\partial f}{\partial x_n}\right)R=\m_R.
\]
\end{cor}

\begin{proof}
This is \eqref{iso:jacobian-image} with $c=1$.
\end{proof}

\subsection{Reduction to a quadratic equation}
\label{subsection:quadratic}

For a quadratic form $q$, its \emph{polar form} is the bilinear form
$b_q(v,w)=q(v+w)-q(v)-q(w)$. We explicitly require nondegeneracy of
this bilinear form; in characteristic two this is stronger than some
other conventions for a nondegenerate quadratic form.

We recall the linear algebra used below. Over an algebraically closed
field, a quadratic form on a vector space of dimension at least two
has a nonzero isotropic vector $v$. If its polar form is nondegenerate,
choose $w$ with $b_q(v,w)=1$. Replacing $w$ by $w-q(w)v$ makes
$q(w)=0$ while preserving $b_q(v,w)=1$. On the span of $v,w$, the
form is then $xy$. This plane is nondegenerate and splits off
orthogonally. Repeating the construction decomposes the form into
hyperbolic planes, with at most a one-dimensional remainder.
In characteristic two the polar form is alternating, so its
nondegeneracy forces even dimension and there is no remainder.
In odd characteristic a one-dimensional remainder, when present,
can be rescaled to $t^2$. These observations justify all the quadratic
normal forms used below.

\begin{lem}\label{iso:formal-morse}
Let $S=k[[x_1,\ldots,x_n]]$, let $f\in S$ have order two, and let $q$ be its
homogeneous quadratic part. If the polar form of $q$ is nondegenerate,
then a formal change of coordinates carries $f$ to $q$.
\end{lem}

\begin{proof}
We write $q_{x_i}:=\partial q/\partial x_i$.
Let $V=k^n$, and let $e_1,\dots,e_n$ be its standard basis.
Since $q$ is homogeneous of degree two, we have
\[
  q(x+t e_i)=q(x)+t\,b_q(x,e_i)+t^2q(e_i).
\]
Comparing the coefficients of $t$ gives
\[
  q_{x_i}(x)=b_q(x,e_i).
\]
The nondegeneracy of $b_q$ implies that the map
\[
  V\longrightarrow V^\vee,\qquad
  v\longmapsto b_q(-,v)
\]
is an isomorphism.
Hence $q_{x_1},\dots,q_{x_n}$ form a basis of $V^\vee$,
the space of linear forms. Suppose inductively that the terms of degrees three
through $N-1$ have been removed, so that
\[
 f=q+h_N+r_{N+1},\qquad
 \deg h_N=N,\quad r_{N+1}\in\m_S^{N+1},\quad N\geq3.
\]
Since the $q_{x_i}$ generate $\m_S$, there are homogeneous polynomials
$v_i$ of degree $N-1$ such that
$h_N=\sum_i q_{x_i}v_i$.
The substitution $x_i\mapsto x_i-v_i$ is an automorphism with identity
linear part. Under this substitution,
\[
 q(x-v)=q(x)-\sum_iq_{x_i}(x)v_i+q(v).
\]
Here $q(v)$ has degree $2N-2$. The change in $h_N$ also has order at
least $2N-2$, and the image of $r_{N+1}$ still belongs to
$\m_S^{N+1}$. Since $2N-2\geq N+1$, the substitution removes $h_N$
without changing any lower-degree term. Iterating, we obtain a
sequence of coordinate changes whose successive corrections have
orders tending to infinity. Their compositions converge
$\m_S$-adically to an automorphism, and the transformed equation
converges to $q$. %No division by two is used.
\end{proof}

\begin{cor}\label{iso:quadratic-reduction}
Let $S=k[[x_1,\ldots,x_n]]$ and $f\in\m_S^2$.
If the normal isolated hypersurface singularity $R=S/(f)$
has dimension at least three and is
$F$-liftable, then it is formally isomorphic to a quadratic
hypersurface whose polar form is nondegenerate.
\end{cor}

\begin{proof}
%We may assume $f\in\m_S^2$. 
By
Corollary~\ref{iso:jacobian-hypersurface}, the partial derivatives of $f$
generate $\m_R$. Since $f\in\m_S^2$, the natural map
$\m_S/\m_S^2\to\m_R/\m_R^2$ is an isomorphism.
The linear parts of the partial derivatives are the derivatives of
the quadratic part $q$ of $f$. They span this $n$-dimensional space,
so the polar form of $q$ is nondegenerate (cf.\ the argument in the proof of Lemma \ref{iso:formal-morse}).
The conclusion follows
from Lemma~\ref{iso:formal-morse}.
\end{proof}

\begin{cor}\label{iso:double-plane}
Suppose $p=2$ and set $S=k[[x_1,x_2,x_3,x_4]]$.
A normal isolated three-dimensional hypersurface singularity $S/(f)$
with an equation
\[
 f=x_1^2+f_{\geq3},\qquad f_{\geq3}\in\m_S^3,
\]
is not $F$-liftable.
\end{cor}

\begin{proof}
All first partial derivatives belong to $\m_S^2$, so
$J_f\subseteq\m_R^2$. This contradicts
Corollary~\ref{iso:jacobian-hypersurface}, because
$\m_R/\m_R^2\ne0$.
\end{proof}

\subsection{Quadratic hypersurfaces in higher dimension}

For odd $p$, the non-$F$-liftability of ordinary double points in
dimension at least four was proved in \cite{Zda18}. The following
calculation also treats characteristic two and supplies the precise
obstruction needed here.

\begin{lem}\label{iso:quadrics}
Let $q\in k[x_1,\ldots,x_n]$ be a quadratic form with nondegenerate
polar form, where $n\geq5$. Then
$k[[x_1,\ldots,x_n]]/(q)$ is not $F$-liftable.
\end{lem}
\begin{proof}
% By the quadratic normal forms recalled at the beginning of Subsection~4.2,
% after a linear change of coordinates we may assume that all nonzero
% coefficients of $q$ are $1$.
% Let $Q$ denote the same polynomial over $\mathbb Z$.
% The definition of $\Delta_1$ then gives
% \[
% \Delta_1(q)
% =\frac{Q(x_1,\ldots,x_n)^p-Q(x_1^p,\ldots,x_n^p)}{p}
% \bmod p.
% \]
% The numerator has integer coefficients divisible by $p$, so the quotient
% is taken over $\mathbb Z$ before reduction.
By the hypersurface criterion of Lemma~\ref{lem:deltaone-criterion}, Frobenius liftability would
imply the existence of a power series $h$ such that
\[
\Delta_1(q)+h^p\in(q,x_1^p,\ldots,x_n^p).
\tag{\thethm.1}\label{eq:deltaone-quad-criterion}
\]
Indeed, the $p$-th powers of the partial derivatives of $q$ generate
$(x_1^p,\ldots,x_n^p)$, by nondegeneracy of the polar form (cf.\ the argument in the proof of Lemma \ref{iso:formal-morse}).

Assume first that $p$ is odd.
We split off two hyperbolic planes as explained at the beginning
of Subsection~\ref{subsection:quadratic}. Since $p\neq2$, their nondegenerate orthogonal
complement can be diagonalized. Since $k$ is algebraically closed,
we can then rescale the remaining coordinates so that
\[
q=xy+zw+t^2+x_6^2+\cdots+x_n^2.
\]
Setting $x_6=\cdots=x_n=0$ in \eqref{eq:deltaone-quad-criterion} gives the corresponding containment
for $q_0=xy+zw+t^2$, since this specialization also sends
$\Delta_1(q)$ to $\Delta_1(q_0)$.
%This is simply a specialization of an ideal containment; no assertion
%about preservation of Frobenius liftability under hyperplane sections
%is required. 
Work in the graded Artinian ring
\[
T=k[x,y,z,w,t]/(x^p,y^p,z^p,w^p,t^p).
\]
Every $p$-th power in $T$ is constant. Comparing constant terms in \eqref{eq:deltaone-quad-criterion}
forces the image of $h^p$ to vanish. Thus we would have
$\Delta_1(q_0)\in q_0T$. Write $A=xy$, $B=zw$, and $C=t^2$, so that
\[
\Delta_1(q_0)
=\frac{(A+B+C)^p-A^p-B^p-C^p}{p}\bmod p.
\]
The element $s=q_0^{p-1}t^{p-3}$ belongs to $T$ and satisfies
$sq_0=q_0^pt^{p-3}=0$. We claim, however, that
$s\Delta_1(q_0)\neq0$. The coefficient of the nonzero monomial
\[
M=x^{p-1}y^{p-1}z^{p-1}w^{p-1}t^{p-1}
\]
in $s\Delta_1(q_0)$ is the coefficient of $A^{p-1}B^{p-1}C$ in
$(A+B+C)^{p-1}\Delta_1(q_0)$.
None of the terms $(A+B+C)^{p-1}A^p$, $(A+B+C)^{p-1}B^p$, or
$(A+B+C)^{p-1}C^p$ contributes to this coefficient.
Computing over $\mathbb Z$ before reduction, the coefficient is therefore
\[
\frac{(2p-1)!}{p((p-1)!)^2}
=\binom{2p-1}{p-1}
=\prod_{j=1}^{p-1}\frac{p+j}{j}
\equiv1\pmod p.
\]
Thus $s\Delta_1(q_0)\neq0$, whereas $sq_0=0$ contradicts
$\Delta_1(q_0)\in q_0T$. 
%The calculation includes $p=3$.

Now assume $p=2$. The polar form is alternating and nondegenerate,
so $n=2r$ is even. Since $n\geq5$, we have $r\geq3$.
Over the algebraically closed field $k$, the quadratic form is a sum
of hyperbolic planes:
\[
q=x_1y_1+\cdots+x_ry_r.
\]
Specializing all but the first three pairs to zero reduces \eqref{eq:deltaone-quad-criterion}
to the corresponding containment for $q_0=xy+zw+uv$.
In
\[
T=k[x,y,z,w,u,v]/(x^2,y^2,z^2,w^2,u^2,v^2),
\qquad A=xy,\quad B=zw,\quad C=uv,
\]
we have $q_0=A+B+C$ and $\Delta_1(q_0)=AB+AC+BC$.
As before, the square $h^2$ is constant in $T$ and vanishes by comparison
of constant terms. Thus $\Delta_1(q_0)\in q_0T$ would be necessary.
But
\[
q_0^2=0,
\qquad q_0\Delta_1(q_0)=ABC=xyzwuv\neq0.
\]
Multiplying the asserted containment by $q_0$ gives a contradiction.
\end{proof}

\subsection{Proof of the classification}

\begin{proof}[Proof of Theorem~\ref{thm:main-ici}]
Suppose $R$ is $F$-liftable. Proposition~\ref{iso:codimension} shows
that its embedding codimension is zero or one. In codimension zero,
$R$ is a regular complete local ring. In codimension one, we have a
minimal hypersurface presentation with at least four variables, and
Corollary~\ref{iso:quadratic-reduction} reduces it to a quadratic form
with nondegenerate polar form. If $d\geq4$, there are at least five
variables, and Lemma~\ref{iso:quadrics} gives a contradiction.
If $d=3$, the quadratic form has four variables. 
% Over an algebraically
% closed field, a quadratic form with nondegenerate polar form on a
% four-dimensional space is a sum of two hyperbolic planes. Hence
Since $k$ is algebraically closed, we have
\[
 R\simeq k[[x,y,z,t]]/(xy-zt).
\]

% Conversely, a regular complete local $k$-algebra with residue field
% $k$ is isomorphic to a power series ring. Its standard lifting over
% $W_2(k)$ carries the Frobenius lifting that acts as $\sigma$ on
% coefficients and sends each variable to its $p$-th power.
% For the ordinary double point, take
% \[
%  \widetilde R=W_2(k)[[x,y,z,t]]/(xy-zt).
% \]
% This ring is flat over $W_2(k)$: the displayed equation reduces to a
% nonzerodivisor in $k[[x,y,z,t]]$, so the local criterion for flatness
% applies. The same rule on coefficients and variables preserves the
% defining ideal, because
% \[
%  x^py^p-z^pt^p=(xy)^p-(zt)^p
% \]
% is divisible by $xy-zt$ over $W_2(k)$. It therefore defines a
% Frobenius lifting of $\widetilde R$.
Conversely, a regular complete local $k$-algebra with residue field
$k$ is a power series ring and hence Frobenius liftable.
The ordinary double point is likewise Frobenius liftable, using
the flat lifting
\[
 \widetilde R=W_2(k)[[x,y,z,t]]/(xy-zt)
\]
with the map acting as $\sigma$ on coefficients and sending each
variable to its $p$-th power.
\end{proof}

\begin{rem}\label{iso:ci-essential}
The complete intersection hypothesis is essential. Let $r\geq2$
be prime to $p$, and let $\mu_r$ act on $\mathbf A_k^d$, $d\geq2$,
by scalar multiplication. The action is free away from the origin,
so the quotient is regular there. Its vertex is singular: the
invariant ring is the $r$-th Veronese subring, whose local embedding
dimension is $\binom{r+d-1}{d-1}>d$.
The completion at the vertex is nevertheless $F$-liftable in every
dimension. Indeed, lift the invariant monomial semigroup algebra
to $W_2(k)$, let Frobenius act as $\sigma$ on coefficients and as
$p$-th power on monomials, and complete at the vertex; see also
\cite{Zda18} for the toric construction.
% In Proposition~\ref{iso:codimension}, the complete intersection
% hypothesis supplies both the two-term free resolution
% \eqref{iso:resolution} and the Gorenstein duality used in
% \eqref{iso:duality}.
\end{rem}

\begin{prop}\label{prop:terminal-cdv-fregular}
Let $X$ be a Gorenstein terminal threefold over an algebraically closed field
$k$ of characteristic $p>0$. Then every singular closed point $P\in X$ is a
compound Du Val singularity, i.e., a general hyperplane section of $X$
through $P$ has a Du Val singularity at $P$. Moreover, if $p>5$, then $X$ is strongly
$F$-regular. In particular, $X$ is $F$-pure when $p>5$.
\end{prop}

\begin{proof}
Let $P\in X$ be a singular closed point, and set
$A=\sO_{X,P}$ and $R=\widehat A$.
By \cite[Theorem~1.2]{Tanaka25}, we have
\[
 R\simeq k[[x,y,z,t]]/(f),\qquad \ord(f)=2.
\]
Since the complete intersection property descends from completion,
$A$ is a complete intersection. Thus, after shrinking around $P$, we may
assume that $X$ is a locally complete intersection.

% For a prime divisor $E$ over $X$ with centre $P$, write $a(E;X)$ for its
% discrepancy and $A_X(E)=a(E;X)+1$ for its log discrepancy.
% Since $K_X$ is Cartier, $a(E;X)$ is an integer. Terminality therefore gives
% \[
%  a(E;X)\geq 1,\qquad A_X(E)\geq 2.
% \]
% For a normal locally complete intersection variety, Mather--Jacobian log
% discrepancies coincide with the usual log discrepancies by
% \cite[Remark~2.5(2)]{IR17}. Consequently,
% \[
%  \mld_{\mathrm{MJ}}(P;X)
%  =\mld(P;X)\geq 2=\dim X-1.
% \]
% By \cite[Theorem~4.8]{IR17}, $(X,P)$ is a compound Du Val
% singularity, since the other singular alternatives in that theorem
% are two-dimensional.
Since $X$ is Gorenstein and terminal, the discrepancy $a(E;X)$
of every prime divisor $E$ over $X$ with centre $P$ is a positive
integer. Hence all log discrepancies at $P$ are at least $2$.
As $X$ is normal and locally a complete intersection,
\cite[Remark~2.5(2) and Theorem~4.8]{IR17} implies that $(X,P)$
is a compound Du Val singularity.

Now assume that $p>5$. By the preceding paragraph, there is a nonzero
element $\ell\in\mathfrak m_R$ such that $R/\ell R$ is a two-dimensional
rational double point. By \cite[Corollary~4.9]{Hara98}, $R/\ell R$ is
strongly $F$-regular. Since $R$ is $F$-finite and Gorenstein,
\cite[Corollary~4.11]{HaraWatanabe02} implies that $R$ is strongly $F$-regular.
Strong $F$-regularity descends along the faithfully flat completion map
$A\to R$, so $A$ is strongly $F$-regular as well. At smooth closed points,
the local rings are regular and hence strongly $F$-regular. Localization
then gives the assertion at every point of $X$.
%Finally, strong $F$-regularity implies $F$-purity.
\end{proof}

\section{A lifting criterion for hypersurfaces of the form \texorpdfstring{$xy-g$}{xy-g}}\label{sec:ca-criterion}

% \begin{notation}\label{ca:setup}
% Throughout this section, we set
% \[
%  B=k[[z,t]],\qquad \widetilde B=W_2(k)[[z,t]],\qquad
%  R_g=B[[x,y]]/(xy-g),
% \]
% where $0\neq g\in(z,t)$. We denote by $\Phi_0$ the endomorphism of
% $\widetilde B$ that restricts to the Witt Frobenius on $W_2(k)$ and sends
% $z$ and $t$ to $z^p$ and $t^p$, respectively. For a lifting
% $G\in\widetilde B$ of $g$, set
% \begin{equation}\label{ca:obstruction}
%  P_g=\frac{\Phi_0(G)-G^p}{p}\bmod p\in B.
% \end{equation}
% As in Lemma~\ref{pre:criterion}, division by $p$ in this expression means
% the inverse of the isomorphism
% $\widetilde B/p\widetilde B\xrightarrow{\sim}p\widetilde B$ induced by
% multiplication by $p$. If $G$ is replaced by $G+p\widetilde h$, then
% $P_g$ is replaced by $P_g+h^p$, where $h=\widetilde h\bmod p$.
% Indeed, $\Phi_0(p\widetilde h)=p\Phi_0(\widetilde h)$ and
% $(G+p\widetilde h)^p=G^p$, since $p^2=0$.

% \end{notation}
\begin{notation}\label{not:deltaone-ca}
For the rest of the paper, we set
\[
B=k[[z,t]],\qquad
\widetilde B=W_2(k)[[z,t]],\qquad
R_g=B[[x,y]]/(xy-g),
\]
where $0\neq g\in(z,t)$.
We denote by $\Phi_0$ the continuous endomorphism of
$\widetilde B$ that restricts to the Witt Frobenius on $W_2(k)$
and sends $z,t$ to $z^p,t^p$, respectively.
We use $\Delta_1$ as defined in Notation~\ref{not:deltaone-definitions}, applied to the
variables $z,t$.
Unless stated otherwise, $G\in\widetilde B$ denotes the
coefficientwise Teichm\"uller lifting of $g$.
Thus
\[
\frac{\Phi_0(G)-G^p}{p}\bmod p=-\Delta_1(g).
\tag{\thethm.1}\label{eq:deltaone-g-defect}
\]
As in Notation~\ref{not:deltaone-definitions}, division by $p$ means the inverse of
the isomorphism $B\simeq p\widetilde B$ induced by
multiplication by $p$.
\end{notation}

\begin{rem}
\label{rem:cA}
The rings above include the completions of all normal
three-dimensional $cA$ singularities in the usual sense,
namely those whose general hyperplane section is of type $A_n$
for some $n\geq 1$.
Indeed, such a section has embedding dimension three, so the
completion is a hypersurface $k[[x,y,z,t]]/(f)$.
The $A_n$ section shows that the polar form of the quadratic
part of $f$ has rank at least two.
Over the algebraically closed field $k$, a hyperbolic plane
therefore splits off, and a linear change of coordinates
puts the quadratic part in the form $xy+q_0(z,t)$.
The degree-by-degree argument in the proof of Lemma~\ref{iso:formal-morse},
applied to $x,y$ while keeping $z,t$ fixed, eliminates all
higher-order terms involving $x$ or $y$, yielding
\[
\widehat{\mathcal O}_{X,P}
\simeq k[[x,y,z,t]]/(xy-g(z,t)).
\]
Here $g\in(z,t)^2$, and normality implies $g\neq 0$.
\end{rem}

\begin{prop}\label{prop:deltaone-ca-criterion}
The following conditions are equivalent.
\begin{enumerate}
\item $R_g$ is $F$-liftable.
\item There exists $h\in B$ such that
\[
\Delta_1(g)+h^p\in(g^p,g_z^p,g_t^p)B.
\tag{\thethm.1}\label{eq:deltaone-ca-criterion}
\]
\item There exist a lifting $G'\in\widetilde B$ of $g$,
a Frobenius lifting $\Phi$ on $\widetilde B$, and an element
$\widetilde c\in\widetilde B$ such that
\[
\Phi(G')=(G')^p(1+p\widetilde c).
\tag{\thethm.2}\label{eq:deltaone-preserved-equation}
\]
\end{enumerate}
\end{prop}

\begin{proof}
We first prove the equivalence of (1) and (2).
Set $S=B[[x,y]]$ and $f=xy-g$.
By the definition of $\Delta_1$ in Notation~\ref{not:deltaone-definitions}, we have
\[
\Delta_1(f)
=-\Delta_1(g)
 +\sum_{i=1}^{p-1}
 \overline{\frac{1}{p}\binom{p}{i}}
 (xy)^i(-g)^{p-i}.
\]
Here the bar denotes reduction modulo $p$ of the indicated
integer coefficient.
% Indeed, the monomial $xy$ is distinct from every monomial of $g$,
% and $\Delta_1(-g)=-\Delta_1(g)$ by homogeneity of degree $p$.
Modulo $(f)$, every term in the sum is a scalar multiple of $g^p$.
Since $g^p\equiv x^py^p\pmod{(f)}$, we obtain
\[
\Delta_1(f)\equiv-\Delta_1(g)
\pmod{(f,x^p,y^p)}.
\tag{\thethm.3}\label{eq:deltaone-ca-congruence}
\]
% This argument is valid also when $p=2$; in that case the sum
% consists of the single term $xyg$, which is congruent to $g^2$
% modulo $(f)$.
Since $f_x=y$, $f_y=x$, $f_z=-g_z$, and $f_t=-g_t$,
Lemma~\ref{lem:deltaone-criterion} and \eqref{eq:deltaone-ca-congruence} show, after replacing
the auxiliary power series by its negative, that $R_g$ is
$F$-liftable if and only if
\[
\Delta_1(g)+H^p
\in(xy-g,x^p,y^p,g_z^p,g_t^p)S
\tag{\thethm.4}\label{eq:deltaone-ca-four-variable-criterion}
\]
for some $H\in S$.

To eliminate $x$ and $y$, consider the finite free $B$-algebra
\[
T=B[[x,y]]/(x^p,y^p)
 =\bigoplus_{0\leq i,j<p}Bx^iy^j.
\]
Give $T$ the $\mathbb Z$-grading determined by
$\deg x=1$, $\deg y=-1$, and $\deg B=0$.
If $s=xy$, then $T_0=B[s]/(s^p)$.
The elements $xy-g$, $g_z^p$, and $g_t^p$ have degree zero,
so the degree-zero component of their quotient is
\[
\begin{aligned}
\bigl(T/(xy-g,g_z^p,g_t^p)\bigr)_0
&\simeq B[s]/(s^p,s-g,g_z^p,g_t^p)\\
&\simeq B/(g^p,g_z^p,g_t^p).
\end{aligned}
\]
% The first isomorphism follows by taking the degree-zero part
% of each coefficient multiplying the three homogeneous generators.
Furthermore, $H^p=H(0,0,z,t)^p$ in $T$, because every term
involving $x$ or $y$ has zero $p$-th power in $T$.
Consequently, \eqref{eq:deltaone-ca-four-variable-criterion}
is equivalent to \eqref{eq:deltaone-ca-criterion}, with
$h=H(0,0,z,t)$; conversely one may take $H=h$.

We next prove that (2) implies (3).
Write
\[
\Delta_1(g)+h^p=cg^p+ag_z^p+bg_t^p,
\qquad a,b,c\in B,
\]
and choose liftings
$\widetilde h,\widetilde a,\widetilde b,\widetilde c
\in\widetilde B$.
Set $G'=G-p\widetilde h$ and define $\Phi$ by the Witt
Frobenius on coefficients and
\[
\Phi(z)=z^p+p\widetilde a,\qquad
\Phi(t)=t^p+p\widetilde b.
\]
Taylor expansion modulo $p^2$ and
\eqref{eq:deltaone-g-defect} give
\[
\begin{aligned}
\frac{\Phi(G')-(G')^p}{p}\bmod p
&=-\Delta_1(g)-h^p+ag_z^p+bg_t^p\\
&=-cg^p.
\end{aligned}
\]
By the identification $B\simeq p\widetilde B$, this is
equivalent to
$\Phi(G')=(G')^p(1-p\widetilde c)$.
Replacing $\widetilde c$ by its negative gives (3).

Conversely, assume (3).
Write $G'=G-p\widetilde h$,
$\Phi(z)=z^p+p\widetilde a$, and
$\Phi(t)=t^p+p\widetilde b$, and denote their reductions
by $h,a,b$, with $c=\widetilde c\bmod p$.
The same expansion of
\eqref{eq:deltaone-preserved-equation} yields
\[
-\Delta_1(g)-h^p+ag_z^p+bg_t^p=cg^p.
\]
Hence
$\Delta_1(g)+h^p=-cg^p+ag_z^p+bg_t^p$, proving (2).
%
% For completeness, (3) also gives a lifting of $R_g$ directly.
% The ring
% \[
% \widetilde R_g=\widetilde B[[x,y]]/(xy-G')
% \]
% is flat over $W_2(k)$, since the reduction of its defining
% equation is a nonzerodivisor in the power series ring over $k$.
% Extend $\Phi$ by
% \[
% \Phi(x)=x^p(1+p\widetilde c),\qquad \Phi(y)=y^p.
% \]
% Then
% \[
% \Phi(xy-G')
% =(1+p\widetilde c)\bigl((xy)^p-(G')^p\bigr)
% \]
% is divisible by $xy-G'$.
% Thus $\Phi$ descends to a Frobenius lifting on
% $\widetilde R_g$.
\end{proof}

\section{Classification in the general \texorpdfstring{$cA$}{cA} case}\label{sec:ca-classification}

\begin{notation}\label{ca:factors}
Write the factorization of $g$ in the unique factorization domain $B$ as
\[
 g=u\prod_{i=1}^r h_i^{n_i},
\]
where $u\in B^\times$ and the $h_i$ are pairwise nonassociated irreducible
elements. We set
\begin{equation}\label{ca:reduced-divisor}
 D_g=\prod_{p\nmid n_i}h_i,
\end{equation}
with $D_g=1$ when the product is empty. Although the element $D_g$ is
defined only up to multiplication by a unit, its zero locus is a
well-defined reduced formal plane curve.
\end{notation}

We will show that $F$-liftability
forces this curve to have at most two smooth branches meeting transversely.

\begin{lem}\label{ca:primitive}
Let $h\in k[[z,t]]$ be irreducible. Then $h_z$ and $h_t$ are not both
divisible by $h$. Consequently, $dh$ is part of a basis of
$\Omega^1_{B/k}\otimes_B B_{(h)}$.
\end{lem}

\begin{proof}
Suppose that both derivatives belong to $(h)$. Then the ideal $(h)$ is
preserved by $\partial_z$ and $\partial_t$: for example,
$\partial_z(ah)=a_zh+ah_z\in(h)$ for every $a\in B$.
Since $k$ is perfect, there is a unique expansion
\[
 h=\sum_{0\leq i,j<p}a_{ij}^pz^it^j,\qquad a_{ij}\in B.
\]
For each pair $(i,j)$, differentiation gives
\begin{equation}\label{ca:coefficient-extraction}
 \partial_z^{p-1}\partial_t^{p-1}
 \bigl(z^{p-1-i}t^{p-1-j}h\bigr)
   =((p-1)!)^2a_{ij}^p.
\end{equation}
Indeed, derivatives of the coefficients $a_{ij}^p$ vanish. Among the
possible exponents between $0$ and $2p-2$, the $(p-1)$-st derivative of a
monomial is nonzero only when its exponent is $p-1$: smaller exponents
give zero, and for larger exponents the product of consecutive integers
in its derivative contains a factor $p$. This proves
\eqref{ca:coefficient-extraction} in both variables.
Its left-hand side belongs to $(h)$, and $(p-1)!$ is nonzero in $k$.
Thus $a_{ij}^p\in(h)$ for all $(i,j)$. Since $(h)$ is prime,
$a_{ij}\in(h)$ for all $(i,j)$, which implies $h\in(h^p)$, a contradiction.
At least one of $h_z,h_t$ is therefore a unit in the discrete valuation
ring $B_{(h)}$. A vector with a unit coefficient can be completed to a
basis of the free module $B_{(h)}dz\oplus B_{(h)}dt$, proving the last
assertion.
\end{proof}

\begin{lem}\label{ca:determinant}
If $R_g$ is $F$-liftable, then
\begin{equation}\label{ca:fedder-condition}
 D_g^{p-1}\notin(z^p,t^p)B.
\end{equation}
\end{lem}

\begin{proof}
Choose $G'$, $\Phi$, and $\widetilde c$ as in
Proposition~\ref{prop:deltaone-ca-criterion}(3), and write $c=\widetilde c\bmod p$.
The divided differential of $\Phi$ defines an additive map
\[
 \zeta:\Omega^1_{B/k}\longrightarrow\Omega^1_{B/k},\qquad
 \zeta(dv)=\frac{d\Phi(\widetilde v)}p\bmod p,
\]
where $\widetilde v$ is any lifting of $v$ and differentials on
$\widetilde B$ are continuous and relative to $W_2(k)$.
The $B$-linear map
$\Omega^1_{B/k}\to F_*\Omega^1_{B/k}$,
$\omega\mapsto F_*\zeta(\omega)$,
is the split injection associated with $\Phi$ in Proposition~\ref{pre:splitting}.
The continuous relative differentials on $\widetilde B$ form a free
$\widetilde B$-module with basis $dz,dt$.
The reduction of $d\Phi(\widetilde v)$ is $d(v^p)=0$, so its division
by $p$ is defined as above. Replacing $\widetilde v$ by
$\widetilde v+p\widetilde w$ changes the divided differential by
$d\Phi(\widetilde w)\bmod p=d(w^p)=0$; hence the definition is independent
of the lifting. The product rule shows that
\[
 \zeta(a\omega)=a^p\zeta(\omega)
 \qquad(a\in B,\ \omega\in\Omega^1_{B/k}).
\]
In particular, $\zeta$ extends to any localization of $B$ by
$\zeta(a^{-1}\omega)=a^{-p}\zeta(\omega)$.

Differentiating \eqref{eq:deltaone-preserved-equation}, dividing by $p$, and reducing
modulo $p$ gives
\begin{equation}\label{ca:differential-compatibility}
 \zeta(dg)=g^{p-1}dg+g^pdc.
\end{equation}
In fact,
\[
 d\bigl((G')^p(1+p\widetilde c)\bigr)
   =p(G')^{p-1}dG'+p(G')^p d\widetilde c,
\]
because every omitted term is divisible by $p^2$.

Let $h$ be an irreducible factor of $g$ with multiplicity $n$ not
divisible by $p$. In $B_{(h)}$ we may write $g=vh^n$, where $v$ is a unit.
Using the Frobenius semilinearity of $\zeta$ in
\eqref{ca:differential-compatibility} and dividing by $h^{p(n-1)}$, we obtain
\[
 h^p\zeta(dv)+nv^p\zeta(dh)
   =h^pv^{p-1}dv+nv^ph^{p-1}dh+h^pv^pdc.
\]
Since $nv^p$ is a unit, it follows that
\begin{equation}\label{ca:normal-direction}
 \zeta(dh)\equiv h^{p-1}dh
 \pmod{h^p\bigl(\Omega^1_{B/k}\otimes_B B_{(h)}\bigr)}.
\end{equation}

Let $M$ be the matrix whose columns are the coefficients of
$\zeta(dz)$ and $\zeta(dt)$ in the basis $dz,dt$.
By Lemma~\ref{ca:primitive}, we can choose a basis of
$\Omega^1_{B/k}\otimes_B B_{(h)}$ whose first element is $dh$.
In this basis, \eqref{ca:normal-direction} implies that the first column
of the matrix of $\zeta$ is divisible by $h^{p-1}$. Its determinant is
therefore divisible by $h^{p-1}$.
To compare this with $M$, let $A\in\operatorname{GL}_2(B_{(h)})$ be the
basis-change matrix. The new matrix of the semilinear map is
$A^{-1}MA^{[p]}$, where $A^{[p]}$ denotes entrywise $p$-th power. Its
determinant is
\[
 (\det A)^{p-1}\det M.
\]
The additional factor is a unit. Thus $h^{p-1}$ divides $\det M$ in
$B_{(h)}$. Applying this to every factor in \eqref{ca:reduced-divisor}
and using unique factorization in $B$, we obtain
\begin{equation}\label{ca:determinant-divisibility}
 D_g^{p-1}\mid\det M\quad\text{in }B.
\end{equation}

It remains to show that $\det M\notin(z^p,t^p)$.
Write
$\Phi(z)=z^p+p\widetilde a$ and
$\Phi(t)=t^p+p\widetilde b$, and let $a,b$ be their reductions.
Then
\[
 \zeta(dz)=z^{p-1}dz+da,\qquad
 \zeta(dt)=t^{p-1}dt+db.
\]
Both forms are closed and have Cartier images $dz$ and $dt$.
By the multiplicativity of the Cartier operator,
\begin{equation}\label{ca:cartier-determinant}
 C\bigl((\det M)dz\wedge dt\bigr)=dz\wedge dt.
\end{equation}
% Explicitly, if
% \[
%  \operatorname{Tr}_B\left(\sum_{i,j\geq0}a_{ij}z^it^j\right)
%  =\sum_{r,s\geq0}a_{pr+p-1,\,ps+p-1}^{1/p}z^rt^s,
% \]
% then \eqref{ca:cartier-determinant} means
% $\operatorname{Tr}_B(\det M)=1$. In particular, the coefficient of
% $z^{p-1}t^{p-1}$ in $\det M$ is $1$.
We write $\operatorname{Tr}_B$ for the additive map defined by
\[
\operatorname{Tr}_B
\left(\sum_{i,j\geq0}a_{ij}z^it^j\right)
=
\sum_{r,s\geq0}
a_{pr+p-1,\,ps+p-1}^{1/p}z^rt^s.
\]
The explicit formula for the Cartier operator gives
\[
C(F\,dz\wedge dt)
=\operatorname{Tr}_B(F)\,dz\wedge dt
\qquad(F\in B).
\]
Thus \eqref{ca:cartier-determinant} gives $\operatorname{Tr}_B(\det M)=1$.
In particular, the coefficient of $z^{p-1}t^{p-1}$ in
$\det M$ is $1$.
This coefficient can also be checked directly: the difference
\[
 \zeta(dz)\wedge\zeta(dt)-z^{p-1}t^{p-1}dz\wedge dt
 =d\bigl(at^{p-1}dt-bz^{p-1}dz+a\,db\bigr)
\]
is exact, and an exact two-form has zero coefficient at
$z^{p-1}t^{p-1}dz\wedge dt$: a monomial contributing to this coefficient
before differentiation would have exponent $p$ in the variable being
differentiated.
Consequently, $\det M\notin(z^p,t^p)$.
Together with \eqref{ca:determinant-divisibility}, this proves
\eqref{ca:fedder-condition}.
\end{proof}

\begin{lem}\label{ca:reduced-curve}
Let $D\in k[[z,t]]$ be a nonzero reduced element. If
$D^{p-1}\notin(z^p,t^p)$, then $V(D)$ is empty, smooth, or a node.
\end{lem}

\begin{proof}
There is nothing to prove if $D$ is a unit. If
$\operatorname{ord}(D)=1$, then $D$ is part of a regular system of
parameters, so $V(D)$ is smooth. Suppose that
$m=\operatorname{ord}(D)\geq2$.
Every monomial of total degree greater than $2(p-1)$ belongs to
$(z^p,t^p)$. Thus $m\geq3$ would imply
$D^{p-1}\in(z^p,t^p)$, a contradiction. We must have $m=2$.

The ideal $(z^p,t^p)$ is the Frobenius power $(z,t)^{[p]}$, and hence
is unchanged by a change of regular parameters. Multiplying $D$ by
a unit also does not affect the membership of $D^{p-1}$ in this ideal.
If the quadratic part of $D$ were a double line, we could therefore
assume
\[
 D=z^2+D_{\geq3}.
\]
The lowest-degree term of $D^{p-1}$ would then be $z^{2p-2}$, which
belongs to $(z^p,t^p)$. Every remaining term has
total degree greater than $2(p-1)$ and belongs to the same ideal.
This is again a contradiction.

Since $k$ is algebraically closed, the quadratic part of $D$
is a product of two distinct linear forms.
After a linear change of coordinates and multiplication by
a nonzero constant, it is $zt$.
Lemma~\ref{iso:formal-morse} then gives formal coordinates in which $D=zt$,
so $V(D)$ is a node.
% Since $k$ is algebraically closed, the quadratic part of $D$ is
% therefore a product of two distinct linear forms. After a linear
% change of variables and multiplication by a nonzero constant, it is
% $zt$. 
% We explain why the higher-order terms can be removed formally.
% Suppose that the current expression is
% $D=zt+D_N+D_{\geq N+1}$, where $N\geq3$ and $D_N$ is homogeneous of
% degree $N$. Write
% \[
%  D_N=zA_{N-1}+tB_{N-1}
% \]
% with $A_{N-1},B_{N-1}$ homogeneous of degree $N-1$.
% The coordinate change
% $z\mapsto z-B_{N-1}$, $t\mapsto t-A_{N-1}$ cancels $D_N$.
% All new error terms have degree at least $2N-2\geq N+1$.
% Iterating and passing to the $(z,t)$-adic limit, we obtain formal
% coordinates in which $D=zt$, so $V(D)$ is a node.
\end{proof}
%\commentbox{taking limit tokani suru?}

\begin{proof}[Proof of Theorem~\ref{thm:main-ca}]
Suppose first that $R_g$ is $F$-liftable. Lemma~\ref{ca:determinant}
gives $D_g^{p-1}\notin(z^p,t^p)$, and Lemma~\ref{ca:reduced-curve}
shows that $V(D_g)$ is empty, smooth, or a node. This proves
$(1)\Rightarrow(3)$.

Assume (3). Choose regular parameters so that the branches occurring
in $D_g$, if any, are coordinate axes. Write each multiplicity in the
factorization of $g$ as $n_i=pq_i+r_i$, where $0\leq r_i<p$.
Every factor with $r_i\neq0$ is associated to $z$ or $t$.
Absorbing the corresponding units into $u$, and setting
$H=\prod_i h_i^{q_i}$, we obtain
\[
 g=uH^pz^at^b,\qquad u\in B^\times,\quad
 0\neq H\in B,\quad 0\leq a,b<p.
\]
This proves $(3)\Rightarrow(2)$. Conversely, in such an expression
every irreducible factor other than $z$ and $t$ has multiplicity
divisible by $p$. The multiplicities of $z$ and $t$ modulo $p$ are
$a$ and $b$, respectively. Hence $V(D_g)$ is a union of a subset of
the two coordinate axes, proving $(2)\Rightarrow(3)$.

Finally, assume (2). With $x'=u^{-1}x$, the equation
$xy-uH^pz^at^b$ becomes $u(x'y-H^pz^at^b)$.
Thus $R_g\cong R_{H^pz^at^b}$, and we may assume $u=1$.
Choose any lifting $\widetilde H\in W_2(k)[[z,t]]$ of $H$ and put
\[
 G=\widetilde H^{\,p}z^at^b.
\]
Since $\Phi_0(\widetilde H)\equiv\widetilde H^{\,p}\pmod p$,
there is an element $A\in\widetilde B$ such that
$\Phi_0(\widetilde H)=\widetilde H^{\,p}+pA$. The identity
$(U+pV)^p=U^p$ in a ring in which $p^2=0$ gives
\[
 \Phi_0(\widetilde H)^p=\widetilde H^{\,p^2},
 \qquad \Phi_0(G)=G^p.
\]
Thus the standard Frobenius lifting on
$W_2(k)[[x,y,z,t]]$ preserves the ideal $(xy-G)$, because
\[
 \Phi_0(xy-G)=(xy)^p-G^p
  =(xy-G)\sum_{i=0}^{p-1}(xy)^{p-1-i}G^i.
\]
It descends to a Frobenius lifting on the flat ring
$W_2(k)[[x,y,z,t]]/(xy-G)$. This proves $(2)\Rightarrow(1)$.
\end{proof}

{
\noindent\textbf{Use of generative AI}

Much of the mathematical content of this paper, including the
statements and proofs of Theorems~\ref{thm:main-ici}
and~\ref{thm:main-ca} and key intermediate results such as
Proposition~\ref{iso:codimension}, was obtained through extensive
interactions with OpenAI's ChatGPT (GPT-6 Astra).
ChatGPT was also used to refine the wording and improve the clarity
of the exposition.
The authors revised and incorporated this material into the
manuscript and take full responsibility for the final
mathematical content and presentation.

}

\section*{Acknowledgements}
Kawakami was supported by JSPS KAKENHI Grant number JP24K16897 and by Inamori Foundation.
Takamatsu was supported by JSPS KAKENHI Grant number JP25K17228.

\bibliographystyle{alpha}
\bibliography{Kawakami-Takamatsu}
\end{document}